\documentclass[12pt,a4paper]{article}
\usepackage{pifont}
\usepackage{amsfonts}
\usepackage{mathrsfs}
\usepackage{amssymb}
\usepackage{amsmath}

\newtheorem{theorem}{Theorem}[section]
\newtheorem{definition}[theorem]{Definition}
\newtheorem{lemma}[theorem]{Lemma}
\newtheorem{corollary}[theorem]{Corollary}
\newtheorem{proposition}[theorem]{Proposition}
\newtheorem{example}[theorem]{Example}
\newenvironment{proof}{{\bf Proof.  }}{$\square$}
\makeatletter

\newcommand{\Rmnum}[1]{\expandafter\@slowromancap\romannumeral #1@}
\makeatother

\begin{document}
\title{Gr\"{o}bner-Shirshov bases for $L$-algebras\footnote{Supported by the
NNSF of China (Nos.10771077, 10911120389) and the NSF of Guangdong
Province (No.06025062).}}
\author{
L. A. Bokut\footnote {Supported by RFBR 01-09-00157, LSS--344.2008.1
and SB RAS Integration grant No. 2009.97 (Russia).} \\
{\small \ School of Mathematical Sciences, South China Normal
University}\\
{\small Guangzhou 510631, P. R. China}\\
{\small Sobolev Institute of Mathematics, Russian Academy of
Sciences}\\
{\small Siberian Branch, Novosibirsk 630090, Russia}\\
{\small  bokut@math.nsc.ru}\\
\\
 Yuqun
Chen\footnote {Corresponding author.} \ and   Jiapeng Huang \\
{\small \ School of Mathematical Sciences, South China Normal
University}\\
{\small Guangzhou 510631, P. R. China}\\
{\small yqchen@scnu.edu.cn}\\
{\small jphuang1985@163.com}}

\date{}

\maketitle \noindent\textbf{Abstract:}    In this paper, we firstly
establish Composition-Diamond lemma for $\Omega$-algebras. We give a
Gr\"{o}bner-Shirshov basis of the free $L$-algebra as a quotient
algebra of a free $\Omega$-algebra, and then the normal form of the
free $L$-algebra is obtained. We secondly establish
Composition-Diamond lemma for $L$-algebras. As applications, we give
 Gr\"{o}bner-Shirshov bases of the free dialgebra and the free product
 of two $L$-algebras, and then we show four embedding theorems of
$L$-algebras: 1) Every countably generated $L$-algebra  can be
embedded into a two-generated $L$-algebra. 2) Every $L$-algebra  can
be embedded into a simple $L$-algebra. 3) Every countably generated
$L$-algebra over a countable field can be embedded into a simple
two-generated $L$-algebra. 4) Three arbitrary $L$-algebras $A$, $B$,
$C$ over a field $k$ can be embedded into a simple $L$-algebra
generated by $B$ and $C$ if $|k|\leq \dim(B*C)$ and $|A|\leq|B*C|$,
where $B*C$ is the free product of $B$ and $C$.

\ \


 \section{Introduction}

The theories of Gr\"obner-Shirshov bases and Gr\"obner bases were
invented independently by A.I. Shirshov  (\cite{Sh}, 1962) for
non-commutative and non-associative algebras, and by H. Hironaka
(\cite{H}, 1964) and B. Buchberger (\cite{bu65}, 1965) for
commutative algebras. Gr\"obner-Shirshov technique is proved to be
very useful in the study of presentations of many kinds of algebras
by generators and defining relations.

V. Drensky and R. Holtkamp \cite{Dren} constructed Gr\"obner bases
theory for algebras with multiple operations ($\Omega$-algebras,
where $\Omega$ consists of $n$-ary operations, $n\geq2$) and proved
the Diamond lemma for $\Omega$-algebras. For associative
$\Omega$-algebras, Gr\"obner-Shirshov bases and Composition-Diamond
lemma were established in \cite{Qiu}.

The class of $L$-algebras was invented by P. Leroux \cite{Ph}. In
\cite{Ph1}, P. Leroux gave a coalgebraic framework of directed
graphs equipped with weights in terms of $L$-coalgebras, which are
$k$-spaces with two co-operations $\Delta_M,~\tilde{\Delta}_M$
verifying the entanglement relation:
$$
(\tilde{\Delta}_M\otimes id)\Delta_M = (id \otimes
\Delta_M)\tilde{\Delta}_M.
$$
$L$-algebras arise from this coding weighted directed graphs
\cite{Ph}. They are algebras over a field $k$ with two operations
$\prec$, $\succ$ satisfied one identity:
$$
(x \succ y)\prec z = x \succ ( y \prec z).
$$
In fact, many types of algebras are $L$-algebras. For example,
associative algebras. In this case, the two operations coincide with
the associative product.

P. Leroux found a normal form of  a free $L$-algebra in \cite{Ph}.

In this paper, we establish Composition-Diamond lemma for
$\Omega$-algebras, where $\Omega$ consists of $n$-ary operations,
$n\geq1$. This generalizes the result in V. Drensky and R. Holtkamp
\cite{Dren}. As a result,  we give another linear basis for the free
$L$-algebra by using Composition-Diamond lemma for
$\Omega$-algebras. Then we continue to study $L$-algebras and
establish the Composition-Diamond lemma for $L$-algebras. As
applications, we prove embedding theorems for $L$-algebras:

1) Every countably generated $L$-algebra over a field $k$ can be
embedded into a two-generated $L$-algebra.

2) Every $L$-algebra over a field $k$ can be embedded into a simple
$L$-algebra.

3) Every countably generated $L$-algebra over a countable field $k$
can be embedded into a simple two-generated $L$-algebra.

4) Three arbitrary $L$-algebras $A$, $B$, $C$ over a field $k$ can
be embedded into a simple $L$-algebra generated by $B$ and $C$ if
$|k|\leq \dim(B*C)$ and $|A|\leq|B*C|$, where $B*C$ is the free
product of $B$ and $C$.

We also give the Gr\"{o}bner-Shirshov bases of a free dialgebra and
the free product of two $L$-algebras, respectively, and then the
normal forms of such algebras are obtained.

\section{Composition-Diamond lemma for
$\Omega$-algebras}

In this section, we establish the Composition-Diamond lemma for
$\Omega$-algebras.

Let $k$ be a field, $X$  a set of variables,  $\Omega$  a set of
multilinear operations, and
$$
\Omega=\cup_{n\geq1}\Omega_n,
$$
where $\Omega_n=\{\delta_{i}^{(n)}|i\in I_n\}$ is the set of $n$-ary
operations, $n=1,2,\dots$. Now, we define ``$\Omega$-words".
(Usually they are called terms in $X$ and $\Omega$, see, for example
\cite{mal}.)

Define
$$
(X,\Omega)_0 = X.
$$

For $m \geqslant 1$, define
$$
(X,\Omega)_m = X \cup \Omega((X,\Omega)_{m-1})
$$
where
$$
\Omega((X,\Omega)_{m-1}) = \cup_{t=1}^\infty \{
\delta_{i}^{(t)}(u_1, u_2, \dots, u_t) |~\delta_{i}^{(t)} \in
\Omega_t, u_j \in\ (X,\Omega)_{m-1}\}.
$$

Let
$$
(X,\Omega)=\bigcup_{m=0}^\infty (X,\Omega)_m.
$$
Then each element in $(X,\Omega)$ is called an $\Omega$-word.

Let $k(X,\Omega)$ be a $k$-linear space with $k$-basis $(
X,\Omega)$. For any $\delta_{i}^{(t)} \in \Omega_t$, extend linearly
$\delta_{i}^{(t)}$ to $k(X,\Omega)$. Then $k(X,\Omega)$ is a free
$\Omega$-algebra generated by $X$. We call the elements of
$k(X,\Omega)$ as $\Omega$-polynomials.

Let $k(X,\Omega)$ be the free $\Omega$-algebra defined as above and
$\star \not\in X$. By a $\star$-$\Omega$-word we mean any expression
in $(X \cup \{\star\},\Omega)$ with only one occurrence of $\star$.
Let $u$ be a $\star$-$\Omega$-word and $s \in k( X,\Omega)$. Then we
call $u|_s = u|_{\star\mapsto s}$ an $s$-$\Omega$-word.

Similar to $\star$-$\Omega$-word, we define $(\star_1,\star_2
)$-$\Omega$-words as expressions in $(X\cup
\{\star_1,\star_2\},\Omega)$ with only one occurrence of $\star_1$
and only one occurrence of $\star_2$.  Let $u$ be a
$(\star_1,\star_2)$-$\Omega$-word and $s_1,~s_2 \in  k(X,\Omega)$.
We call
$$
u|_{s_1,s_2} = u|_{\star_1\mapsto s_1,\star_2\mapsto s_2}
$$
 an $s_1$-$s_2$-$\Omega$-word.

Let $>$ be a well ordering on $(X,\Omega)$. Then $>$ is called
monomial if for any $\star$-$\Omega$-word $w$ and any $ u, v\in
(X,\Omega)$, $u>v$ implies $w|_u>w|_v$. The following example shows
that such a monomial ordering on $(X,\Omega)$ exists.

\begin{example}
Suppose that $X$ and $\Omega$ are well-ordered sets. For any $u\in
(X,\Omega)$, if $u=x\in X$, define $wt(u)=(1,x)$; if
$u=\delta^{(i)}(u_1,u_2,\dots,u_i)$, define
$wt(u)=(|u|_{\Omega}+|u|_X, |u|_X,\delta^{(i)},u_1,u_2,\dots,u_i)$,
where $|u|_T$ means the number of $t\in T$ in $u$. For any $u,v\in
(X,\Omega)$, we define
$$
u>v\Longleftrightarrow wt(u)>wt(v) ~~~~~~~lexicographically
$$
by induction on $|u|_{\Omega}+|u|_X+|v|_{\Omega}+|v|_X$. Then it is
easy to see that $>$ is a monomial ordering on $(X,\Omega)$.
\end{example}

\noindent \textbf{Remark}: \ In V. Drensky and R. Holtkamp
\cite{Dren}, they use the ordering (\ref{1}) (see next section).
However, the ordering (\ref{1}) is not well-ordered on $(X,\Omega)$.
For example, let $\delta^{(1)},\zeta^{(1)}\in \Omega_1$ with
$\delta^{(1)}>\zeta^{(1)}$ and $x\in X$. Then for the ordering
(\ref{1}), we have an infinite descending chain
$$
\delta^{(1)}(x)>\zeta^{(1)}(\delta^{(1)}(x))>\zeta^{(1)}(\zeta^{(1)}(\delta^{(1)}(x)))>\cdots.
$$

\ \

From now on, in this section, we assume that $(X,\Omega)$ is
equipped with a monomial ordering $>$.

For any $\Omega$-polynomial $0\neq f\in k(X,\Omega)$, let $\bar{f}$
be the leading $\Omega$-word of $f$. If the coefficient of $\bar{f}$
is $1$, then $f$ is called monic.

\begin{definition}
Let $f, g$ be two monic $\Omega$-polynomials. If there exists an
$\Omega$-word $w = \bar{f} =
 u|_{\bar{g}}$ for some $\star$-$\Omega$-word $u$, then we call
 $(f,g)_w = f- u|_g$ the inclusion composition of $f$ and $g$
 with respect to $w$. If this is the case, $w$ is called the
 ambiguity of the composition $(f,g)_w$.
 \end{definition}

 \begin{definition}
 Let $S$ be a set of monic $\Omega$-polynomials. Then the composition
 $(f,g)_w $ is called trivial modulo $(S, w)$, denoted by $(f,g)_w \equiv 0~mod(S,
 w)$, if
 $$
 (f,g)_\omega = \Sigma
 \alpha_iu_i|_{s_i},
 $$
where each $\alpha_i \in k,\ u_i\ \star$-$\Omega$-word, $s_i \in S$
 and $u_i|_{\bar{s_i}}<w$.

 $S$ is called a Gr\"{o}bner-Shirshov basis in
 $k(X,\Omega)$ if for any $f, g \in
 S$,  $(f,g)_w \equiv 0~mod(S,
 w)$.
 \end{definition}

A subset $I$ of $k(X,\Omega)$ is called an $\Omega$-ideal of
$k(X,\Omega)$ if $I$ is a subspace such that for any
$\star$-$\Omega$-word $u$,
$$
u|_{_I}=\{u|_{_f}|f\in I\}\subseteq I.
$$

 \begin{theorem}\label{CDL}
 (Composition-Diamond lemma for $\Omega$-algebras)\
 Let $S$ be a set of monic $\Omega$-polynomials in $k(X,\Omega)$,
 $>$ a monomial ordering on $(X,\Omega)$ and $Id(S)$
the $\Omega$-ideal of $k(X,\Omega)$ generated by $S$. Then the
following statements are equivalent:
\begin{enumerate}
\item[ \Rmnum{1} )] $S $ is a Gr\"{o}bner-Shirshov basis in $k(X,\Omega)$.
\item[ \Rmnum{2} )] $f\in Id (S)\Rightarrow \bar{f} = u|_{\bar{s}}$
where $u|_s$ is an $s$-$\Omega$-word, $s \in S$.
\item[ \Rmnum{3})] $Irr (S) = \{ \omega \in (X,\Omega)|  \omega \neq
u|_{\bar{s}}$, $u|_s$ is an $s$-$\Omega$-word, $s \in S\}$ is a
$k$-basis of the algebra $k( X,\Omega|S) = k( X,\Omega)/Id(S)$.
\end{enumerate}
\end{theorem}
\begin{proof}
$\Rmnum{1})\Rightarrow \Rmnum{2})$. \ Let $S$ be a
Gr\"{o}bner-Shirshov basis in $k(X,\Omega)$ and $0\neq f\in Id(S)$.
We can assume that
$$
f=\sum_{i=1}^n\alpha_iu_i|_{s_i},
$$
where each $\alpha_i\in k,  \ s_i\in S$ and $u_i|_{s_i}$
$s_i$-$\Omega$-word. Let
$$
w_i=u_i|_{\overline{s_i}}, \
w_1=w_2=\cdots=w_{_l}>w_{_{l+1}}\geq\cdots
$$

We prove the result by induction on $l$ and $w_1$.

If $l=1$, then
$\overline{f}=\overline{u_1|_{s_1}}=u_1|_{\overline{s_1}}$ and the
result holds. Assume that $l\geq 2$. Then
$$
\alpha_1u_1|_{s_1}+\alpha_2u_2|_{s_2}=(\alpha_1+\alpha_2)u_1|_{s_1}+\alpha_2(u_2|_{s_2}-u_1|_{s_1}).
$$

There are two cases to consider:

1) $\overline{s_1}$ and $\overline{s_2}$ are disjoint in $w_1$. Then
there exists a $(\star_1,\star_2)$-$\Omega$-word $\Pi$ such that
$$
\Pi|_{\overline{s_1},\overline{s_2}}=u_1|_{\overline{s_1}}=u_2|_{\overline{s_2}}
$$
and
$$
u_1|_{s_1}-u_2|_{s_2}=\Pi|_{s_1,\overline{s_2}}-\Pi|_{\overline{s_1},s_2}
=-\Pi|_{s_1,s_2-\overline{s_2}}+\Pi|_{s_1-\overline{s_1},s_2}.
$$

Since $ \overline{s_2-\overline{s_2}}<\overline{s_2}$ and
$\overline{s_1-\overline{s_1}}<\overline{s_1}$, we have
$\overline{\Pi|_{s_1,s_2-\overline{s_2}}}<w_1$ and $
\overline{\Pi|_{s_1-\overline{s_1},s_2}}<w_1.$ Noting that
$\Pi|_{s_1,s_2-\overline{s_2}}$ is linear combination of
$s_1$-$\Omega$-words and $\Pi|_{s_1-\overline{s_1},s_2}$ is linear
combination of $s_2$-$\Omega$-words, we have
$$
u_1|_{s_1}\equiv u_2|_{s_2}~mod(S,w_1).
$$

Thus, if $\alpha_1+\alpha_2\neq 0$ or $l>2$, then the result follows
from the induction on $l$. For the case $\alpha_1+\alpha_2= 0$ and
$l=2$, we use the induction on $w_1$. The result follows.

2) One of $\overline{s_1}$, $\overline{s_2}$ is contained in the
other. We may assume that $\overline{s_2}$ is contained in
$\overline{s_1}$, i.e., $\overline{s_1}=u|_{\overline{s_2}}$ for
some $s_2$-$\Omega$-word $u|_{s_2}$. Thus,
$$
u_1|_{s_1}-u_2|_{s_2}=u_1|_{s_1}-u_1|_{u|_{s_2}}=u_1|_{s_1-u|_{s_2}}.
$$
Since S is a Gr\"{o}bner-Shirshov basis in $k(X,\Omega)$, we have
$$
s_1-u|_{s_2}=\sum_t \alpha_tv_t|_{s_t}
$$
where each $\alpha_t\in k$, $s_t\in S$, $v_t|_{s_t}$
$s_t$-$\Omega$-word, and $v_t|_{\overline{s_t}}<\overline{s_1}$. So
$u_1|_{v_t|_{\overline{s_t}}}<w_1$ for any $t$. Now, the result
follows.

$\Rmnum{2})\Rightarrow \Rmnum{3})$. For any $f\in k( X,\Omega)$, we
may express $f$ as
$$
f=\sum_{u_i\in Irr(S),\ u_i\leqslant
\overline{f}}\alpha_iu_i+\sum_{s_j\in S,\
u_j|_{\overline{s_j}}\leqslant \overline{f}}\beta_ju_j|_{s_j},
$$
where $\alpha_i,\beta_j\in k$ and $u_j|_{s_j}$ $s_j$-$\Omega$-word.
So any $f\in k(X,\Omega)$ can be expressed modulo $Id(S)$ as a
linear combination of elements from $Irr(S)$. That is, $Irr(S)$
spans $k(X,\Omega|S)$ as $k$-space.

Suppose $g=\alpha_1u_1+\alpha_2u_2+\cdots+\alpha_nu_n=0$ in $k(
X,\Omega|S)$, where $u_i\in Irr(S)$,  $\alpha_i\neq 0$,
$i=1,2,\dots,n$ and $u_1>u_2>\cdots>u_n$. Then in $k( X,\Omega)$,
$g\in Id(S)$. By $\Rmnum{2})$,
$\overline{g}=u_1=u|_{\overline{s}}\notin Irr(S)$, a contradiction.

$\Rmnum{3})\Rightarrow \Rmnum{1})$. For any composition $(f,g)_w$ in
$S$, by $\Rmnum{3})$,
$$
(f,g)_w=\sum_{s_j\in S,\ u_j|_{\overline{s_j}}\leqslant
\overline{(f,g)_w}}\beta_ju_j|_{s_j}.
$$
Since $\overline{(f,g)_w}<w$, $(f,g)_w\equiv 0\ \ mod(S,w)$.
 \ \hfill\end{proof}

\ \

\noindent \textbf{Remark}: \ When $\Omega$ only contains $n$-ary
multilinear operations for $n\geqslant 2$, the situation is just the
same as that established by V. Drensky and R. Holtkamp in
\cite{Dren} .

\ \

\section{Gr\"{o}bner-Shirshov bases for free $L$-algebras}

In this section, by using the Composition-Diamond lemma for
$\Omega$-algebras (Theorem \ref{CDL}), we give a
Gr\"{o}bner-Shirshov basis of the free $L$-algebra and then a
$k$-basis of such an algebra is obtained.

\begin{definition}\cite{Ph}\label{lalg}
An L-algebra is a $k$-space $L$ equipped with two binary $k$-linear
operations $\prec,~\succ : L{\otimes}L \rightarrow L$ verifying the
so-called entanglement relation:
\begin{equation}\label{0}
(x \succ y)\prec z = x \succ ( y \prec z), \ \ \forall  x, y ,z \in
L
\end{equation}
\end{definition}
Thus, an $L$-algebra is an $\Omega$-algebra  where $\Omega=\{ \prec,
\succ \}$.

In the following, we always assume that $\Omega=\{ \prec, \succ \}$.

Let $k(X,\Omega)$ be the free $\Omega$-algebra generated by $X$. Let
$$
S={\{(x\succ y)\prec z-x\succ (y\prec z)|~x,\ y,\ z\in
(X,\Omega)}\}.
$$
Then $L(X)=k(X,\Omega|S)=k(X,\Omega)/Id(S)$ is clearly a free
$L$-algebra generated by $X$.

We will order the set $(X,\Omega)$.

Let $X$ be a well-ordered set. Denote $|u|_X$ by $|u|$, $\succ$ by
$\delta_1$, and $\prec$ by $\delta_2$. Let $\delta_1<\delta_2$. For
any $u \in (X,\Omega)$, if $u=x\in X$, denote by
$$
wt(u)=(1, x);
$$
if $u=\delta_i(u_1,u_2)$ for some $u_1,~u_2\in (X,\Omega)$, denote
by
$$
wt(u)=(|u|, \delta_i, u_1, u_2).
$$
For any $u,v\in (X,\Omega)$, define
\begin{equation}\label{1}
u>v\Longleftrightarrow wt(u)>wt(v) ~~~~~~~lexicographically
\end{equation} by
induction on $|u|+|v|$.

It is clear that $>$ is a monomial ordering on $(X,\Omega)$.

We will use the ordering (\ref{1}) on $(X,\Omega)$ in sequel.

\begin{theorem}\label{GSBl}
With the ordering (\ref{1}) on $(X,\Omega)$,
$$
S={\{(x\succ y)\prec z-x\succ (y\prec z)|~x,\ y,\ z\in (X,\Omega)}\}
$$
is a Gr\"{o}bner-Shirshov basis in $k(X,\Omega)$.
\end{theorem}
\begin{proof}
All the possible ambiguities of compositions of $\Omega$-polynomials
in $S$ are:
$$
i)~(x|_{(a\succ b)\prec c}\succ y)\prec z~~~~~ ii)~(x\succ
y|_{(a\succ b)\prec c})\prec z~~~~~ iii)~(x\succ y)\prec z|_{(a\succ
b)\prec c}~~~~~
$$
 where $a, b, c, x, y, z\in (X,\Omega)$. It is easy to check that all these compositions
are trivial. Here, for example, we just check $i)$. Others are
similarly proved. Let
$$
f(x,y,z)=(x\succ y)\prec z-x\succ (y\prec z).
$$
 Then
\begin{eqnarray*}
&&(f(x|_{(a\succ b)\prec c},y,z),f(a,b,c))_{(x|_{(a\succ b)\prec
c}\succ y)\prec z}\\
&=& -x|_{(a\succ b)\prec c}\succ (y\prec z)+(x|_{a\succ (b\prec
c)}\succ y)\prec z\\
&\equiv& -x|_{a\succ (b\prec c)}\succ (y\prec z)+x|_{a\succ (b\prec
c)}\succ (y\prec z)\\
& \equiv&0~~mod(S, (x|_{(a\succ b)\prec c}\succ y)\prec z).
\end{eqnarray*}
 \ \hfill
\end{proof}

An $\Omega$-word $u$ is a normal word if $u$ is one of the
following:
\begin{enumerate}
\item[i)]\  $u=x$, \ where $x\in X$.
\item[ii)]\ $u=v\succ w$, where $v$ and $w$ are normal words.
\item[iii)]\ $u=v\prec w$ with $v\neq v_1\succ v_2$, where $v_1,\ v_2,\ v,\ w$ are normal words.
\end{enumerate}

We denote $u$ by $[u]$ if $u$ is a normal word.

\ \

 Now, by Theorem
\ref{CDL},  we have the following corollary.

\begin{corollary}\label{cor1}
The set
$$
Irr (S) = \{ u \in (X,\Omega) |~u \neq v|_{(a\succ b)\prec c},~a,\
b,\ c\in (X,\Omega),\ v\ \mbox{is a}
\star\mbox{-}\Omega\mbox{-}\mbox{word}\}
$$
 is a $k$-basis of the  free L-algebra  $L(X)=k(X,\Omega|S)$. Moreover, $Irr (S)$ consists of all
 normal words in $(X,\Omega)$.
\end{corollary}

The following proposition follows from (\ref{0}) and Corollary
\ref{cor1}.
\begin{proposition}\label{p1}
For any $\Omega$-word $u$,  there exists a unique normal word $[u]$
such that $u=[u]$ in $L(X)$.
\end{proposition}

We denote the set of all the normal words by $N$, i.e., $N=Irr (S)$.
Then, the free $L$-algebra has an expression
$L(X)=kN=\{\sum\alpha_iu_i~|~\alpha_i\in k,\ u_i\in N\}$ with a
$k$-basis $N$ and the operations $\prec,\ \succ$: for any $u,v\in
N$,
$$
u\prec v=[u\prec v], \ \ \ \ u\succ v=[u\succ v].
$$
Clearly, $[u\succ v]=u\succ v$ and
$$
[u\prec v]=\left\{
\begin{array} {ll} u\prec v &\mbox{if } u=u_1\prec u_2,\mbox{or } u\in X,\\
u_1\succ[u_2\prec v] &\mbox{if } u=u_1\succ u_2.
\end{array} \right.
$$

\section{Composition-Diamond lemma for $L$-algebras}

In this section, we establish Composition-Diamond lemma for
$L$-algebras. Remind that $\Omega=\{ \prec, \succ \}$.

We use still the ordering (\ref{1}) on $N$ defined as before. 

Let $u$ be a $\star$-$\Omega$-word and $s \in L(X)$. Then we call
$u|_s = u|_{\star\mapsto s}$ an $s$-word in $L(X)$. Let $u$ be a
$(\star_1,\star_2)$-$\Omega$-word and $s_1,~s_2 \in L(X)$. We call
$$
u|_{s_1,s_2} = u|_{\star_1\mapsto s_1,\star_2\mapsto s_2}
$$
 an $s_1$-$s_2$-word.

An $s$-word $u|_s$ is called a normal $s$-word if
$u|_{\overline{s}}\in N$.

It is noted that the $s$-word $u|_{{s}}$ is a normal
$s$-word if and only if $\overline {u|_{s}}=u|_{\overline{s}}$
as $\Omega$-words.

We will prove that the ordering (\ref{1}) on $N$ is monomial in the
sense that for any $\star$-$\Omega$-word $w$ and any $u,\ v\in N$,
$u>v$ implies $[w|_u]>[w|_v]$.

\begin{lemma}
The ordering (\ref{1}) on $N$ is monomial.
\end{lemma}
\begin{proof}
To prove this lemma, we only need to prove that, for any $u,v,w\in
N$, $u>v$ implies $[u\succ w]>[v\succ w]$, $[w\succ u]>[w\succ v]$,
$[w\prec u]>[w\prec v]$ and $[u\prec w]>[v\prec w]$. But only the
final case needs to check, since the other three cases are just
obvious.

It is clear that $u$ has a unique expression:
$$
u=u_1\succ(u_2\succ( \cdots \succ(u_{n-1}\succ u_n)\cdots))
$$
where $n\geq 1$ and $u_n\neq a\succ b$ for any $a,b\in N$. For
example, if $u=a\prec b$, then $n=1$. Let
$$
v=v_1\succ(v_2\succ( \cdots \succ(v_{m-1}\succ v_m)\cdots))
$$
where $m\geq 1$ and $v_m\neq a\succ b$ for any $a,b\in N$. Then
$$
[u\prec w]=u_1\succ(u_2\succ( \cdots \succ(u_{n-1}\succ (u_n\prec
w))\cdots)),
$$
$$
[v\prec w]=v_1\succ(v_2\succ( \cdots \succ(v_{m-1}\succ (v_m\prec
w))\cdots)).
$$

If $|u|>|v|$, $[u\prec w]>[v\prec w]$ is obvious. We assume that
$|u|=|v|$. Since $u>v$, there must be a $j$ such that  $u_j>v_j$,
and for any $i<j$, $u_i=v_i$. So
$$
(u_{j}\succ( \cdots \succ(u_{n-1}\succ (u_n\prec
w))\cdots)))>(v_j\succ( \cdots \succ(v_{m-1}\succ (v_m\prec
w))\cdots))).
$$
It follows that $[u\prec w]>[v\prec w]$.
 \ \hfill \end{proof}

\ \

Now we define compositions of polynomials in $L(X)$.

\begin{definition}
Let $f, g \in L(X)$ with $f$ and $g$ monic.
\begin{enumerate}
\item[1)] \ Composition of right multiplication.

If $\overline{f}=u_1\succ u_2$ for some $u_1,\ u_2\in N$, then for
any $v\in N$, $f\prec v$ is called the composition of right
multiplication.

\item[2)] \  Composition of inclusion.

If $w=\overline{f}=u|_{\overline{g}}$ where $u|_{g}$ is a normal
$g$-word, then
$$
(f,g)_{w}=f-u|_g
$$
is called a composition of inclusion.
\end{enumerate}
\end{definition}

\begin{definition}
Let $S\subset L(X)$ be a monic set and $f,g\in S$.

The composition of right multiplication $f\prec v$ is called trivial
modulo $S$, denoted by $ f\prec v\equiv 0 \ mod(S), $ if
$$
f\prec v=\sum{\alpha}_iu_i|_{s_i},
$$
where each $\alpha_i\in k, \ s_i\in S, \ u_i|_{s_i}$ normal
$s_i$-word, and $u_i|_{\overline{s_i}}\leqslant \overline{f\prec
v}$.

The composition of inclusion $(f,g)_{w}$ is called trivial modulo
$(S, w)$, denoted by $(f,g)_{w}\equiv 0 \ mod(S,w)$, if
$$
(f,g)_{w}=\sum{\alpha}_iu_i|_{s_i},
$$
where each $\alpha_i\in k, \ s_i\in S, \ u_i|_{s_i}$ normal
$s_i$-word, and $u_i|_{\overline{s_i}}<w$.

$S$ is called a Gr\"{o}bner-Shirshov basis in $L(X)$ if any
composition of polynomials in $S$ is trivial.
\end{definition}

\begin{lemma}\label{le0}
Let $S\subset L(X)$ and $u|_s$ an $s$-word, $s\in S$. Assume
that each composition of right multiplication in $S$ is trivial
modulo $S$. Then, $u|_s$ has a presentation:
$$
u|_s=\sum\alpha_iu_i|_{s_i},
$$
where each $\alpha_i\in k, \ s_i\in S$, $u_i|_{s_i}$ normal
$s_i$-word and $u_i|_{\overline{s_i}}\leqslant \overline{u|_s}$.
\end{lemma}

\begin{proof}
By Proposition \ref{p1}, we may assume that $u=[u]\in N$. We prove
the result by induction on $|u|$.

If $|u|=1$, $u|_s=s$ is a normal $s$-word.

If $|u|>1$, we have four cases to consider:

1) $u|_s=w\succ v|_s$,

2) $u|_s=v|_s\succ w$,

3) $u|_s=w\prec v|_s$,

4) $u|_s=v|_s\prec w$.

Since $w$ is a subword of the normal word $u$, $w$ is also an normal
word. Since $|v|<|u|$, $v|_s=\sum\alpha_ju_j|_{s_j}$, where each
$\alpha_j\in k, \ s_j\in S$, $u_j|_{s_j}$ normal $s_j$-word and
$u_j|_{\overline{s_j}}\leqslant \overline{v|_s}$. Moreover, we may
assume that $v|_s$ is a normal $s$-word.

In Case 1), 2) and 3), $\Omega$-word $u|_{\overline{s}}\in N$ and
there is nothing to prove. We only consider Case 4).

In Case 4), if $|v|=1$ and $\overline{s}=a\prec b$ for some $a,b\in
N$, then $\Omega$-word  $u|_{\overline{s}}\in N$. If $|v|=1$ and
$\overline{s}=a\succ b$, then by right multiplication, $u|_s=s\prec
w=\sum\alpha_iu_i|_{s_i}$, where each $\alpha_i\in k, \ s_i\in S,\
u_i|_{s_i}$ normal $s_i$-word and
$u_i|_{\overline{s_i}}\leqslant \overline{u|_s}$.

For $|v|>1$, since $u=[u]\in N$, $v=a\prec b$ for some $a,b\in N$.
Then by $v|_{\overline{s}}\in N$ we have
$u|_{\overline{s}}=v|_{\overline{s}}\prec w\in N$ as $\Omega$-words.
So, $u|_s$ is a normal $s$-word.

The lemma is proved.
 \ \hfill \end{proof}

\begin{lemma}\label{le}
Let $S$ be a Gr\"{o}bner-Shirshov basis in $L(X)$, $s_1,\ s_2\in S$,
$u_1|_{s_1}$ and $u_2|_{s_2}$ be normal $s_1$-word and normal
$s_2$-word respectively such that $w
=u_1|_{\overline{s_1}}=u_2|_{\overline{s_2}}$. Then,
$$
u_1|_{s_1}\equiv u_2|_{s_2} \ \ mod(S,w),
$$
where $u_1|_{s_1}\equiv u_2|_{s_2} \ \ mod(S,w)$ means $u_1|_{s_1}-
u_2|_{s_2}=\sum{\alpha}_iu_i|_{s_i}$ for some $\alpha_i\in k, \
s_i\in S, \ u_i|_{s_i}$ normal $s_i$-word such that
$u_i|_{\overline{s_i}}<w$.
\end{lemma}

\begin{proof}
Since the operations $\prec$ and $\succ$ are not associative, there
are only two cases to consider.

1)~$\overline{s_1}$ and $\overline{s_2}$ are disjoint in $w$. Then
there exists a $(\star_1,\star_2)$-$\Omega$-word $\Pi$ such that
$$
\Pi|_{\overline{s_1},\overline{s_2}}=u_1|_{\overline{s_1}}=u_2|_{\overline{s_2}},
$$
where $\Pi|_{s_1,s_2}$ is an  $s_1$-$s_2$-word. Then

$$
u_1|_{s_1}-u_2|_{s_2}=\Pi|_{s_1,\overline{s_2}}-\Pi|_{\overline{s_1},s_2}
=-\Pi|_{s_1,s_2-\overline{s_2}}+\Pi|_{s_1-\overline{s_1},s_2}.
$$

Let
$$
\Pi|_{s_1-\overline{s_1},s_2}=\sum_r \alpha_{2r}u_{2r}|_{s_2},
$$
$$
-\Pi|_{s_1,s_2-\overline{s_2}}=\sum_t \alpha_{1t}u_{1t}|_{s_1}.
$$
Since $S$ is a Gr\"{o}bner-Shirshov basis, by Lemma \ref{le0}, we
have
$$
u_{2r}|_{s_2}=\sum_n\beta_{rn}v_{rn}|_{s_{_{rn}}} \ \ \ \ \
\mbox{and} \ \ \ \ \
u_{1t}|_{s_1}=\sum_m\beta_{tm}v_{tm}|_{s_{_{tm}}},
$$
where  $\beta_{rn},\ \beta_{tm}\in k$, $s_{rn},\ s_{tm}\in S$,
$v_{rn}|_{s_{_{rn}}}$ normal $s_{rn}$-word,
$v_{tm}|_{s_{_{tm}}}$ normal $s_{tm}$-word, and
$v_{tm}|_{\overline{s_{_{tm}}}}\leqslant \overline{u_{1t}|_{s_1}}$,
$v_{rn}|_{\overline{s_{_{rn}}}}\leqslant \overline{u_{2r}|_{s_2}}$.
Then
$$
u_1|_{s_1}-u_2|_{s_2}=\sum_{t,m}\alpha_{1t}\beta_{tm}v_{tm}|_{s_{_{tm}}}+
\sum_{r,n}\alpha_{2r}\beta_{rn}v_{rn}|_{s_{_{rn}}}.
$$
Since
\begin{eqnarray*}
&&v_{tm}|_{\overline{s_{_{tm}}}}\leqslant
\overline{u_{1t}|_{s_1}}\leqslant
\overline{\Pi|_{s_1,s_2-\overline{s_2}}}=\Pi|_{\overline{s_1},
\overline{s_2-\overline{s_2}}}<\Pi|_{\overline{s_1},\overline{s_2}}=w
\ \ \mbox{ and }\\
&&v_{rn}|_{\overline{s_{_{rn}}}}\leqslant
\overline{u_{2r}|_{s_2}}\leqslant
\overline{\Pi|_{s_1-\overline{s_1},s_2}}=\Pi|_{\overline{s_1-\overline{s_1}},
\overline{s_2}}<\Pi|_{\overline{s_1},\overline{s_2}}=w,
\end{eqnarray*}
we have
$$
u_1|_{s_1}\equiv u_2|_{s_2}~mod(S,w).
$$

2)~One of $\overline{s_1}$, $\overline{s_2}$ is contained in the
other. We may assume that $\overline{s_2}$ is contained in
$\overline{s_1}$. Then $\overline{s_1}=u|_{\overline{s_2}}$ for some
normal $s_2$-word $u|_{s_2}$. So
$$
w=u_1|_{\overline{s_1}}=u_1|_{u|_{\overline{s_2}}}
$$
and
$$
u_1|_{s_1}-u_2|_{s_2}=u_1|_{s_1}-u_1|_{u|_{s_2}}=u_1|_{s_1-u|_{s_2}}.
$$
Since S is a Gr\"{o}bner-Shirshov basis in $L(X)$, we have
$$
s_1-u|_{s_2}=\sum_t \alpha_tv_t|_{s_t}
$$
where each $\alpha_t\in k$, $s_t\in S$, $v_t|_{s_t}$ normal
$s_t$-word, and $v_t|_{\overline{s_t}}<\overline{s_1}$. Let
$u_1|_{v_t|_{s_t}}=u_{1t}|_{s_t}$. Then by Lemma \ref{le0}, we have
$$
u_{1t}|_{s_t}=\sum_n\beta_{tn}v_{tn}|_{s_{_{tn}}}
$$
where each $\beta_n\in k$, $s_{tn}\in S$, $v_{tn}|_{s_{_{tn}}}$
normal $s_{tn}$-word, and
$v_{tn}|_{\overline{s_{_{tn}}}}\leqslant \overline{u_{1t}|_{s_t}}$.
So
$$
u_1|_{s_1}-u_2|_{s_2}=\sum_{t}\alpha_{t}u_1|_{v_t|_{s_t}}=\sum_{t}\alpha_{t}u_{1t}|_{s_t}
=\sum_{t,n}\alpha_{t}\beta_{tn}v_{tn}|_{s_{_{tn}}}
$$
with
$$
v_{tn}|_{\overline{s_{_{tn}}}}\leqslant
\overline{u_{1t}|_{s_t}}<u_1|_{\overline{s_1}}=w.
$$
It follows that
$$
u_1|_{s_1}\equiv u_2|_{s_2}~mod(S,w).
$$
The proof is completed. \ \hfill \end{proof}


\begin{theorem}\label{cd}
(Composition-Diamond lemma for L-algebras) \  Let $S\subset L(X)$ be
a monic set and the ordering $>$ defined on $N$ as (\ref{1}). Then
the following statements are equivalent:
\begin{enumerate}
\item[\Rmnum{1})] \ $S$ is a Gr\"{o}bner-Shirshov basis in $L(X)$.
\item[\Rmnum{2})] \
$f\in Id(S)\Rightarrow \overline{f}=u|_{\overline{s}}$ for some
normal $s$-word $u|_s$, $s\in S$.
\item[\Rmnum{3})] \
The set
$$
Irr(S)=\{u\in N|~u\neq v|_{\overline{s}},\ s\in S,\ v|_{s}\mbox{ is
a normal s-word}\}
$$
is a linear basis of the L-algebra $L(X|~S)=L(X)/Id(S)$, where
$Id(S)$ is the ideal of $L(X)$ generated by $S$.
\end{enumerate}
\end{theorem}

\begin{proof}
$\Rmnum{1})\Rightarrow \Rmnum{2})$. \ Let $S$ be a
Gr\"{o}bner-Shirshov basis and $0\neq f\in Id(S)$. By Lemma
\ref{le0}, we can assume that
$$
f=\sum_{i=1}^n\alpha_iu_i|_{s_i},
$$
where each $\alpha_i\in k,  \ s_i\in S$ and $u_i|_{s_i}$ normal
$s_i$-word. Let
$$
w_i=u_i|_{\overline{s_i}}, \
w_1=w_2=\cdots=w_{_l}>w_{_{l+1}}\geq\cdots
$$

We prove the theorem by induction on $l$ and $w_1$.

If $l=1$, then
$\overline{f}=\overline{u_1|_{s_1}}=u_1|_{\overline{s_1}}$ and the
result holds.

Assume that $l\geq 2$. Then
$$
\alpha_1u_1|_{s_1}+\alpha_2u_2|_{s_2}=(\alpha_1+\alpha_2)u_1|_{s_1}+\alpha_2(u_2|_{s_2}-u_1|_{s_1}).
$$
By Lemma \ref{le}, we have
$$
u_2|_{s_2}\equiv u_1|_{s_1} \ \ mod(S,w_1).
$$

Now, the remainder proof of the theorem is almost the same as one in
Theorem \ref{CDL}.
 \ \hfill \end{proof}

\section{Applications}

In this section, we use the Theorem \ref{cd} to prove four embedding
theorems for $L$-algebras. We give Gr\"{o}bner-Shirshov bases of a
free dialgebra and the free product of two $L$-algebras,
respectively and then the norm forms are obtained for such algebras.

Let $X$ be a set. Then we denote the set of all the normal words in
the free $L$-algebra $L(X)$ defined as before by $N(X)$.

Denote by $\mathbb{N}^+$ the set of all positive natural numbers and
$\mathbb{N}=\mathbb{N}^+\cup\{0\}$.

The following lemma is straightforward.

\begin{lemma}\label{eml}
Let $A$ be an $L$-algebra over a field $k$  with a $k$-basis $X =
\{x_i|i \in I\}$. Then $A$ has a representation $A =L(X|S)$, where
$S = \{x_i\prec x_j - \{x_i\prec x_j\},\ x_i\succ x_j - \{x_i\succ
x_j\}|i,j \in I\}$, $\{x_i\succ x_j\}$ and $\{x_i\prec x_j\}$ are
linear combinations of $x_t\in X$. Moreover, with the ordering
(\ref{1}) on $N(X)$, $S$ is a Gr\"{o}bner-Shirshov basis in $L(X)$.
\end{lemma}

\subsection{Gr\"{o}bner-Shirshov basis for the free
product of two $L$-algebras}
\begin{definition}\label{d5.9}
Let $L_1,L_2$ be $L$-algebras. Then an $L$-algebra $L_1*L_2$ with
two $L$-algebra homomorphisms \  $\varepsilon_i: L_i\rightarrow
L_1*L_2$, $i=1,2$ is called the free product of $L_1$ and $L_2$, if
the following diagram commutes:

\setlength {\unitlength}{1cm}
\begin{picture}(7, 3)
\put(4.2,2.5){\vector(1,0){1.7}}\put(6.7,2.3){\vector(0,-1){1.9}}
\put(9.2,2.5){\vector(-1,0){1.7}}\put(4.1, 2.3){\vector(1,-1){2}}
\put(9.2,2.3){\vector(-1,-1){2}}
 \put(6.5,0){$L$}\put(6.9,1.3){$\exists ! f$}
\put(6,2.4){$L_1*L_2$} \put(3.5,2.4){$L_1$} \put(9.5,2.4){$L_2$}
\put(4.9, 2.6){$\varepsilon_1$} \put(8.3,
2.6){$\varepsilon_2$}\put(4.5,0.9){$\forall
f_1$}\put(8.7,0.9){$\forall f_2$}
\end{picture}\\
where $L$ is any $L$-algebra and $f_1,f_2$ are $L$-algebra
homomorphisms. It means that
$(L_1*L_2,(\varepsilon_1,\varepsilon_2))$ is a universal arrow in
the sense of S. Maclane \cite{SM}.
\end{definition}

For $L$-algebras $L_1$ and $L_2$, let $X=\{x_i|i\in I\}$ and
$Y=\{y_j|i\in J\}$ are  $k$-bases of $L_1$ and $L_2$, respectively.
Then $L_1 =L(X|S_1)$ and $L_2 =L(Y|S_2)$, where
\begin{eqnarray*}
&&S_1 = \{x_i\prec x_j-\{x_i\prec x_j\},\ x_i\succ x_j-\{x_i\succ
x_j\}|i,\ j \in I\},\\
&&S_2 = \{y_i\prec y_j-\{y_i\prec y_j\},\ y_i\succ y_j-\{y_i\succ
y_j\}|i,\ j \in J\}.
\end{eqnarray*}
It is clear that
$L_1*L_2=L(X_1\cup X_2|S_1\cup S_2)$.

Let $X\cup Y$ be a well-ordered set. Let $S=S_1\cup S_2\cup F_1\cup
F_2$, where
\begin{eqnarray*}
F_1&=&\{f_{1(ij)}=x_i\succ ((((x_j\prec u)\prec v_1)\prec\cdots)
\prec v_n)-\\
&&\ \ (((\{x_i\succ x_j\}\prec u)\prec v_1)\prec\cdots)\prec v_n\ |
\ n\in \mathbb{N},\ u,v_l\in N(X\cup Y),\\
&&\ \ u\in Irr(S_1\cup S_2)-X,\ v_l\in Irr(S_1\cup S_2),\ l=1,\dots, n\},\\
F_2&=&\{f_{2(ij)}=y_i\succ ((((y_j\prec u)\prec v_1)\prec\cdots)
\prec v_n)-\\
&&\ \ (((\{y_i\succ y_j\}\prec u)\prec v_1)\prec\cdots)\prec v_n\ |\
n\in
\mathbb{N},\ u,v_l\in N(X\cup Y),\\
&&\ \ u\in Irr(S_1\cup S_2)-Y,\ v_l\in Irr(S_1\cup S_2),\ l=1,\dots,
n\}.
\end{eqnarray*}

Then we have the following theorem.

\begin{theorem}\label{t10}
With the ordering (\ref{1}) on $N(X\cup Y)$, $S$ is a
Gr\"{o}bner-Shirshov basis of $L_1*L_2=L(X\cup Y|S_1\cup S_2)$.
\end{theorem}
\begin{proof}All the possible ambiguities $w$ of composition of inclusion in $S$
is $\overline{f_{1(ij)}}|_{\overline{f'_{1(i'j')}}}$,
$\overline{f_{2(ij)}}|_{\overline{f'_{1(i'j')}}}$,
$\overline{f_{1(ij)}}|_{\overline{f'_{2(i'j')}}}$ and
$\overline{f_{2(ij)}}|_{\overline{f'_{2(i'j')}}}$, where
$\overline{f'_{1(i'j')}}$ ($\overline{f'_{1(i'j')}}$,
$\overline{f'_{2(i'j')}}$, $\overline{f'_{2(i'j')}}$ respectively)
is a subword of $u$ or some $v_r$ in $\overline{f_{1(ij)}}$
($\overline{f_{2(ij)}}$, $\overline{f_{1(ij)}}$,
$\overline{f_{2(ij)}}$ respectively).

All the possible compositions of  right multiplication are:
$(x_i\succ x_j-\{x_i\succ x_j\})\prec u$, $(y_i\succ y_j-\{y_i\succ
y_j\})\prec u$, $f_{1(ij)}\prec u$ and $f_{2(ij)}\prec u$ where
$u\in N(X\cup Y)$.

All the compositions are trivial by a similar analysis in Theorem
\ref{em2}.
 \ \hfill \end{proof}

For any $u\in N(X\cup Y)$, we have $u\in Irr(S)$ if and only if  $u$
is one of the following:
\begin{eqnarray*}
&\Rmnum{1})& u\in X\cup Y, \\
&\Rmnum{2})& u\in(X\prec Y)\cup (X\succ Y)\cup (Y\prec X)\cup(
Y\succ X), \mbox{ where, for example, }\\
&&\ \ \  X\prec Y=\{x\prec y|x\in X,\ y\in Y\},\\
&\Rmnum{3})& \mbox{for }|u|>2\mbox{, there are four cases:}\\
&&1)\ u=x_i\succ w\mbox{, with }w\neq ((x_j\prec
w_1)\prec\cdots)\prec
w_n\mbox{, where }x_i,x_j\in X,\ w,w_t\in Irr(S),\\
&&2)\ u=y_i\succ w\mbox{, with }w\neq ((y_j\prec
w_1)\prec\cdots)\prec
w_n\mbox{, where }y_i,y_j\in Y,\ w,w_t\in Irr(S),\\
&&3)\ u=v\succ w\mbox{, with }|v|\geqslant 2\mbox{, where }v,w\in Irr(S),\\
&&4)\ u=v\prec w\mbox{, with }v\neq v_1\succ v_2\mbox{, where
}v,w,v_1,v_2\in Irr(S).
\end{eqnarray*}
\begin{corollary}
By Theorem \ref{cd}, $Irr(S)$ is a $k$-basis of $L_1*L_2$.
\end{corollary}

\subsection{Gr\"{o}bner-Shirshov basis for a free
dialgebra}

\begin{definition}\label{l2}
Let $k$ be a field. A $k$-linear space $D$ equipped with two
bilinear multiplications $\succ $  and $\prec $  is called a
dialgebra, if both $\succ$ and $\prec$ are associative and
\begin{eqnarray*}
a\prec(b\succ c)&=&a\prec b\prec c \\
(a\prec b)\succ c&=&a\succ b\succ c \\
a\succ(b\prec c)&=&(a\succ b)\prec c
\end{eqnarray*}
for any $a, \ b, \ c\in D$.
\end{definition}

Let $D(X)$ be the free dialgebra generated by $X$. Then it is clear
that $D(X)$ is also an $L$-algebra and $D(X)=L(X|S)$, where $S$
consists of
\begin{eqnarray*}
&&F_1=\{a\prec (b\prec c)-a\prec(b\succ c)\ |\ a,b,c\in N(X)\},\\
&&F_2=\{(a\prec b)\succ c-a\succ (b\succ c)\ |\ a,b,c\in N(X)\},\\
&&F_3=\{(a\prec b)\prec c-a\prec (b\succ c)\ |\ a,b,c\in N(X)\},\\
&&F_4=\{(a\succ b)\succ c - a\succ (b\succ c)\ |\ a,b,c\in N(X)\}.
\end{eqnarray*}
Denote by
\begin{eqnarray*}
F_5&=&\{f_{5(n)}=a_1\prec(a_2\succ(a_3\succ \cdots
\succ(a_{n+2}\prec a_{n+3})\cdots))-a_1\prec(a_2\succ(a_3\succ
\cdots \succ(a_{n+2}\\&&\succ a_{n+3})\cdots))\ |\ n\in
\mathbb{N}^+,\ a_i\in N(X),\ i=1,2,\dots,n+3\}.
\end{eqnarray*}

Equipping with the above concepts, we have the following theorem.

\begin{theorem}\label{t5.6}
Let $X$ be a well-ordered set. With the ordering (\ref{1}) on
$N(X)$, $S_1=S\cup F_5$ is a Gr\"{o}bner-Shirshov basis in $L(X)$.
\end{theorem}
\begin{proof}
The possible compositions of  right multiplication are: $f\prec u$,
where $f\in F_2 \cup F_4$, $u\in N(X)$. If $f\in F_2$, we have
\begin{eqnarray*}
~~f\prec u &=&((a\prec b)\succ c)\prec u-(a\succ (b\succ c))\prec u\\
&\equiv&(a\prec b)\succ (c\prec u)-a\succ((b\succ c)\prec u))\\
&\equiv&a\succ (b\succ (c\prec u))-a\succ (b\succ (c\prec u))\\
&\equiv& 0 \ \ mod(S).
\end{eqnarray*}
Similarly, $f\prec u\equiv 0 \ \ mod(S)$ where $f\in F_4$.

For compositions of inclusion, if some $\overline{f_i}$ is a subword
of $a$ ($b$, $c$, $a_i$ respectively) in some $\overline{f_j}$, the
composition $(f_j,f_i)_{\overline{f_j}}$ is  trivial by a similar
proof in Theorem \ref{GSBl}.

All the other ambiguities $w$ of compositions of inclusion in $S_1$
are:
\begin{enumerate}
\item[1)]\ $w=a\prec (b\prec (c\prec
d))=a\prec\overline{f_1}=\overline{f'_1},\mbox{ where }f_1,f'_1\in
F_1,$

\item[2)]\ $w=(a\prec (b\prec c))\succ d=\overline{f_1}\succ
d=\overline{f_2},\mbox{ where }f_1\in F_1,\ f_2\in F_2,$

\item[3)]\ $w=a\prec ((b\prec c)\prec d)=a\prec
\overline{f_3}=\overline{f_1},\mbox{ where }f_1\in F_1,\ f_3\in
F_3,$

\item[4)]\ $w=(a\prec (b\prec c))\prec d=\overline{f_1}\prec
d=\overline{f_3},\mbox{ where }f_1\in F_1,\ f_3\in F_3,$

\item[5)]\ $w=((a\prec b)\prec c)\succ d=\overline{f_3}\succ
d=\overline{f_2},\mbox{ where }f_2\in F_2,\ f_3\in F_3,$

\item[6)]\ $w=((a\prec b)\succ c)\succ d=\overline{f_2}\succ
d=\overline{f_4},\mbox{ where }f_2\in F_2,\ f_4\in F_4,$

\item[7)]\ $w=((a\prec b)\prec c)\prec d=\overline{f_3}\prec
d=\overline{f'_3},\mbox{ where }f_3,f'_3\in F_3,$

\item[8)]\ $w=((a\succ b)\succ c)\succ d=\overline{f_4}\succ
d=\overline{f'_4},\mbox{ where }f_4,f'_4\in F_4,$

\item[9)]\ $w=a_1\prec(a_2\succ(a_3\succ \cdots \succ(a_{i+2}\prec
(a_{i+3}\prec a_{i+4}))\cdots))=a_1\prec(a_2\succ(a_3\succ \cdots
\succ(\overline{f_1})\cdots))=\overline{f_{5(i)}},\mbox{ where
}f_1\in F_1,\ f_{5(i)}\in F_5,$

\item[10)]\ $w=a\prec(a_1\prec(a_2\succ(a_3\succ \cdots \succ(a_{i+2}\prec
a_{i+3})\cdots)))=a\prec\overline{f_{5(i)}}=\overline{f_1},\mbox{
where }f_1\in F_1,\ f_{5(i)}\in F_5,$

\item[11)]\  $w=(a_1\prec(a_2\succ(a_3\succ \cdots \succ(a_{i+2}\prec
a_{i+3})\cdots)))\succ c=\overline{f_{5(i)}}\succ
c=\overline{f_2},\mbox{ where }f_2\in F_2,\ f_{5(i)}\in F_5,$

\item[12)]\  $w=a_1\prec(a_2\succ(a_3\succ \cdots \succ((a\prec b)\succ(a_{i+2}\prec
a_{i+3})\cdots)))=a_1\prec(a_2\succ(a_3\succ \cdots
\succ(\overline{f_2}\prec
a_{i+3})\cdots))=\overline{f_{5(i)}},\mbox{ where }f_2\in F_2,\
f_{5(i)}\in F_5,$

\item[13)]\ $w=(a_1\prec(a_2\succ(a_3\succ \cdots \succ(a_{i+2}\prec
a_{i+3})\cdots)))\prec c=\overline{f_{5(i)}}\prec
c=\overline{f_3},\mbox{ where }f_3\in F_3,\ f_{5(i)}\in F_5,$

\item[14)]\ $w=a_1\prec(a_2\succ(a_3\succ \cdots \succ((a_{i+2}\prec
a_{i+3})\prec a_{i+4})\cdots))=a_1\prec(a_2\succ(a_3\succ \cdots
\succ\overline{f_3})\cdots)=\overline{f_{5(i)}},\mbox{ where }f_3\in
F_3,\ f_{5(i)}\in F_5,$

\item[15)]\ $w=a_1\prec(a_2\succ(a_3\succ \cdots \succ((a\succ b)\succ(a_{i+2}\prec
a_{i+3})\cdots)))=a_1\prec(a_2\succ(a_3\succ \cdots
\succ(\overline{f_4}\prec
a_{i+3})\cdots))=\overline{f_{5(i)}},\mbox{ where }f_4\in F_4,\
f_{5(i)}\in F_5,$

\item[16)]\ $w=a_1\prec(a_2\succ \cdots \succ(a_{i+1}\succ(b_1\prec(b_2\succ \cdots \succ(b_{j+2}\prec
b_{j+3})\cdots))\cdots))=a_1\prec(a_2\succ \cdots \succ(a_{i+1}\succ
\overline{f_{5(j)}})\cdots)=\overline{f_{5(i)}},\mbox{ where
}f_{5(i)}, f_{5(j)}\in F_5$.
\end{enumerate}

These compositions are all trivial. Here, for example, we just check
1), 9) and 16). Others are similarly proved.

1) Let $w=a\prec (b\prec (c\prec
d))=a\prec\overline{f_1}=\overline{f'_1}$. We have
\begin{eqnarray*}
~~(f'_1,f_1)_{w} &\equiv&a\prec (b\succ (c\prec d))-a\prec (b\prec
(c\succ d))\\ &\equiv&a\prec (b\succ (c\prec d))-a\prec (b\succ
(c\succ d))\\&\equiv&f_{5(1)}\\&\equiv& 0\ \ mod(S_1,w).
\end{eqnarray*}

9) Let $w=a_1\prec(a_2\succ(a_3\succ \cdots \succ(a_{i+2}\prec
(a_{i+3}\prec a_{i+4}))\cdots))=a_1\prec(a_2\succ(a_3\succ \cdots
\succ(\overline{f_1})\cdots))=\overline{f_{5(i)}}$. We have
\begin{eqnarray*}
(f_1,f_{5(i)})_{w} &\equiv& a_1\prec(a_2\succ(a_3\succ \cdots
\succ(a_{i+2}\succ (a_{i+3}\prec
a_{i+4}))\cdots))\\&~~&-a_1\prec(a_2\succ(a_3\succ \cdots
\succ(a_{i+2}\prec (a_{i+3}\succ
a_{i+4}))\cdots))\\
&\equiv&a_1\prec(a_2\succ(a_3\succ \cdots \succ(a_{i+2}\succ
(a_{i+3}\prec a_{i+4}))\cdots))\\ &~~&-a_1\prec(a_2\succ(a_3\succ
\cdots \succ(a_{i+2}\succ
(a_{i+3}\succ a_{i+4}))\cdots))\\
&\equiv& f_{5(i+1)}\\ &\equiv& 0\ \ mod(S_1,w).
\end{eqnarray*}

16) Let $w=a_1\prec(a_2\succ \cdots
\succ(a_{i+1}\succ(b_1\prec(b_2\succ \cdots \succ(b_{j+2}\prec
b_{j+3})\cdots))\cdots))=a_1\prec(a_2\succ \cdots \succ(a_{i+1}\succ
\overline{f_{5(j)}})\cdots)=\overline{f_{5(i)}}$. We have
\begin{eqnarray*}
(f_{5(i)},f_{5(j)})_{w} &\equiv& a_1\prec(a_2\succ \cdots
\succ(a_{i+1}\succ(b_1\succ(b_2\succ \cdots \succ(b_{j+2}\prec
b_{j+3})\cdots))\cdots))\\&~~&-a_1\prec(a_2\succ \cdots
\succ(a_{i+1}\succ(b_1\prec(b_2\succ \cdots \succ(b_{j+2}\succ
b_{j+3})\cdots))\cdots))\\
&\equiv& 0\ \ mod(S_1,w).
\end{eqnarray*}

Then $S_1$ is a Gr\"{o}bner-Shirshov basis in $L(X)$.
 \ \hfill \end{proof}

 \ \

The following corollary follows from Theorems \ref{cd} and
\ref{t5.6}.
\begin{corollary}
The set $Irr (S_1) =\{ u=x_{-m}\succ(x_{-m+1}\succ \cdots
\succ(x_{0}\prec (x_{1}\succ \cdots \succ(x_{n-1}\succ
x_{n})\cdots)))|\ n,m\in \mathbb{N}, x_i\in X \}$ is a $k$-basis of
the free dialgebra $D(X)=L(X|S)$.
\end{corollary}

\noindent \textbf{Remark}: \ In \cite{Lo99}, Loday gives a $k$-basis
of the free dialgebra $D(X)$, which is $\{x_{-m}\succ \cdots \succ
x_{-1} \succ x_0 \prec x_1 \prec \cdots \prec x_n|n,m\in \mathbb{N},
x_i\in X\}$. It is easy to see that our $k$-basis is just the same
as in  \cite{Lo99} by using relation $a\prec (b\prec
c)=a\prec(b\succ c)$.

\subsection{Embedding theorems for $L$-algebras}

 \begin{theorem}\label{em2}
Every countably generated $L$-algebra over a field $k$ can be
embedded into a two-generated $L$-algebra.
\end{theorem}

\begin{proof}
Let $A$ be a countably generated $L$-algebra. We may assume that $A$
has a countable $k$-basis $X = \{x_i|i\in \mathbb{N}^+\}$. By Lemma
\ref{eml}, $A =L(X|S)$, where $S = \{x_i\prec x_j-\{x_i\prec x_j\},\
x_i\succ x_j-\{x_i\succ x_j\}|\ i,\ j \in \mathbb{N}^+\}$.

Let $Y=\{a,b\}$,
\begin{eqnarray*}
F_1&=&\{f_{1(ij)}=x_i\succ x_j-\{x_i\succ x_j\}\ |\ i,j\in\mathbb{N}^+\},\\
F_2&=&\{f_{2(ij)}=x_i\prec x_j-\{x_i\prec x_j\}\ |\ i,j\in\mathbb{N}^+\},\\
F_3&=&\{f_{3(i)}=a\prec\underbrace{(b\prec(\cdots (b\prec b))}_i)-x_i\ |\ i,j\in\mathbb{N}^+\},\\
F_4&=&\{f_{4(ij)}=x_i\succ ((((x_j\prec u)\prec
v_1)\prec\cdots)\prec v_n )-\\
&&\ \ (((\{x_i\succ x_j\}\prec u)\prec v_1)\prec\cdots)\prec v_n\ |\
i,j\in\mathbb{N}^+,\
n\in \mathbb{N},\ u,v_l\in N(X\cup Y),\\
&&\ \ u\in Irr(F_1\cup F_2\cup F_3)-X, v_l\in Irr(F_1\cup F_2\cup
F_3),\ l=1,\dots, n\}.
\end{eqnarray*}

Let $B=L(X\cup Y|S_1)$, where $S_1=F_1\cup F_2\cup F_3\cup F_4$.

We want to prove that $S_1$ is a Gr\"{o}bner-Shirshov basis in
$L(X\cup Y)$ with ordering (\ref{1}), where $X\cup Y$ is a
well-ordered set.

 The only ambiguity
$w$ of composition of inclusion in $S_1$ is
$\overline{f_{4(ij)}}|_{\overline{f'_{4(i'j')}}}$ where
$\overline{f'_{4(i'j')}}$ is a subword of $u$ or some $v_l$. It is
trivial by a similar proof in Theorem \ref{GSBl}.

For right multiplication, only possible compositions of right
multiplication are $f_{m(ij)}\prec u$, $m=1,4$, $u\in N(X\cup Y)$.
We can use $F_1$, $F_2$ and $F_3$ to reduce $u$ into a $v\in
Irr(F_1\cup F_2\cup F_3)$ such that $v\leq u$ and
$$
f_{m(ij)}\prec u\equiv f_{m(ij)}\prec v\ \ mod(S_1).
$$
So $f_{m(ij)}\prec u$ is trivial if $f_{m(ij)}\prec v$ is trivial.
Then we only need to consider right multiplication with $u\in
Irr(F_1\cup F_2\cup F_3)$.

We consider the case of $f_{1(ij)}$ first. For any $u\in Irr(F_1\cup
F_2\cup F_3)$, if $u\notin X$, according to $f_{4(ij)}$,
$$
f_{1(ij)}\prec u=(x_i\succ x_j)\prec u- \{x_i\succ x_j\}\prec
u\equiv 0\ \ mod(S_1).
$$
If $u=x_k\in X$, then
\begin{eqnarray*}
f_{1(ij)}\prec u&=&(x_i\succ x_j)\prec u- \{x_i\succ x_j\}\prec
u\\&=&x_i\succ (x_j\prec x_k)-\{x_i\succ x_j\}\prec
x_k\\&\equiv&x_i\succ \{x_j\prec x_k\}-\{x_i\succ x_j\}\prec
x_k\\&\equiv&\{x_i\succ \{x_j\prec x_k\}\}-\{\{x_i\succ x_j\}\prec
x_k\}\\
&\equiv& 0\ \ mod(S_1).
\end{eqnarray*}

Now we consider the case of $f_{4(ij)}$. For any $w\in Irr(F_1\cup
F_2\cup F_3)$,
\begin{eqnarray*}
f_{4(ij)}\prec w&=&(x_i\succ(((x_j\prec u)\prec\cdots)\prec v_n
))\prec w-(((\{x_i\succ x_j\}\prec u)\prec\cdots)\prec v_n )\prec
w\\&=&x_i\succ((((x_j\prec u)\prec\cdots)\prec v_n )\prec
w)-(((\{x_i\succ x_j\}\prec u)\prec\cdots)\prec v_n
)\prec w\\
&\equiv& 0\ \ mod(S_1).
\end{eqnarray*}
So $S_1$ is a Gr\"{o}bner-Shirshov basis in $L(X\cup Y)$. By Theorem
\ref{cd}, $A$ can be embedded into $B$ which is generated by $Y=\{a,
b\}$.
 \ \hfill \end{proof}

  \begin{theorem}
Every $L$-algebra over a field $k$ can be embedded into a simple
$L$-algebra.
\end{theorem}

\begin{proof}
Let $A$ be an $L$-algebra over  $k$ with a $k$-basis $X = \{x_i|i\in
I\}$. By Lemma \ref{eml}, $A =L(X|S)$, where $S = \{x_i\prec
x_j-\{x_i\prec x_j\},\ x_i\succ x_j-\{x_i\succ x_j\}|i,\ j \in I\}$.
Let $I$ be a well-ordered set. Then with the ordering (\ref{1}), $S$
is clearly a Gr\"{o}bner-Shirshov basis in $L(X)$.

We well order the set of monic elements of $A$. Denote by $T$ the
set of indices for the resulting well-ordered set. Consider the set
$T^2=\{(\theta,\sigma)\}$ and assign
$(\theta,\sigma)<(\theta',\sigma')$ if either $\theta<\theta'$ or
$\theta=\theta'$ and $\sigma<\sigma'$. Then $T^2$ is also a
well-ordered set.

For each ordered pair of elements $s_\theta, s_\sigma\in\ A,\
\theta, \sigma\in T$, introduce the letters
$x_{\theta\sigma},y_{\theta\sigma}$.

Let $A_1$ be the $L$-algebra generated by
$$
X_1=\{x_i, y_{\theta\sigma}, x_{\varrho\tau} | i\in I, \ \theta,
\sigma, \varrho, \tau\in T\}
$$
and define the relation set $S_1=F_1\cup F_2\cup F_3\cup F_4$, where
\begin{eqnarray*}
F_1&=&\{x_i\succ x_j-\{x_i\succ x_j\}| i,\ j \in
I\},\\
F_2&=&\{x_i\prec x_j-\{x_i\prec x_j\}| i,\ j \in
I\},\\
F_3&=&\{x_{\theta\sigma}\prec(s_\theta\succ
y_{\theta\sigma})-s_\sigma| ({\theta,\sigma})\in T^2\},\\
F_4&=&\{x_i\succ ((((x_j\prec u)\prec v_1)\prec\cdots)\prec v_n
)-(((\{x_i\succ x_j\}\prec u)\prec v_1)\prec\cdots)\prec v_n\ |\\
&&\ \ \ i,\ j \in I,\ n\in \mathbb{N},\ u,v_l\in L(X_1),\ u\in
Irr(F_1\cup F_2\cup F_3)-X,\\
&&\ \ \  v_l\in Irr(F_1\cup F_2\cup F_3),\ l=1,\dots, n\}.
\end{eqnarray*}
Then, by a similar proof of Theorem \ref{em2}, $S_1$ is a
Gr\"{o}bner-Shirshov basis in $L(X_1)$ with the ordering (\ref{1}),
where $x_i<y_{\theta\sigma}<x_{\varrho\tau}$.

Thus by Theorem \ref{cd}, $A$ can be embedded into $A_1$. In ${A}_1$
every monic element $f_\theta$ of the subalgebra ${A}$ generates an
ideal containing algebra $A$.

By the same construction of the $L$-algebra $A_1$ from $A$, we get
the $L$-algebra $A_2$ from $A_1$ and so on. As a result, we acquire
an ascending chain of $L$-algebras $ A=A_0\subset A_1\subset
A_2\subset\cdots$. Let $ {\cal A}=\cup_ {k=0}^ {\infty}A_k$. In
${\cal A}$, every two nonzero elements generates the same ideal.
Then ${\cal A}$ is a simple $L$-algebra.
 \ \hfill \end{proof}

\begin{theorem}\label{em4}
Every countably generated $L$-algebra over a countable field $k$ can
be embedded into a simple two-generated $L$-algebra.
\end{theorem}

\begin{proof}
Let $A$ be a countably generated $L$-algebra over a countable field
$k$. We may assume that $A$ has a countable $k$-basis $X_0 =
\{x_i|i\in \mathbb{N}^+\}$. By Lemma \ref{eml}, $A =L(X_0|S_0)$,
where $S_0 = \{x_i\prec x_j-\{x_i\prec x_j\},\ x_i\succ
x_j-\{x_i\succ x_j\}|i,\ j \in \mathbb{N}^+\}$.

Let $A_0=L(X_0)$, $A_0^+=A_0-\{0\}$ and fix the bijection
$$
(A_0^+,A_0^+)\leftrightarrow\{(x_m^{(1)},y_m^{(1)})|m\in
\mathbb{N}^+\}.
$$

Let $X_1=X_0\cup\{x_m^{(1)},y_m^{(1)},a,b|m\in \mathbb{N}^+\}$,
$A_1=L(X_1)$, $A_1^+=A_1-\{0\}$ and fix the bijection
$$
(A_1^+,A_1^+)\leftrightarrow\{(x_m^{(2)},y_m^{(2)})|m\in
\mathbb{N}^+\}.
$$

For $n\geqslant 1$, let
$X_{n+1}=X_n\cup\{x_m^{(n+1)},y_m^{(n+1)}|m\in \mathbb{N}^+\}$,
$A_{n+1}=L(X_{n+1})$, $A_{n+1}^+=A_{n+1}-\{0\}$ and fix the
bijection
$$
(A_{n+1}^+,A_{n+1}^+)\leftrightarrow\{(x_m^{(n+2)},y_m^{(n+2)})|m\in
\mathbb{N}^+\}.
$$

Consider the chain of the free $L$-algebras
$$
A_0\subset A_1\subset\cdots\subset A_n\subset\cdots.
$$

Let $X=\cup_{n=0}^{\infty}X_n$. Then $L(X)=\cup_{n=0}^{\infty}A_n$.

Define
\begin{eqnarray*}
F_1&=&\{f_{1(ij)}=x_i\succ x_j-\{x_i\succ x_j\}| i,j\in  \mathbb{N}^+\},\\
F_2&=&\{f_{2(ij)}=x_i\prec x_j-\{x_i\prec x_j\}| i,j\in  \mathbb{N}^+\},\\
F_3&=&\{f_{3(i)}=a\prec\underbrace{(b\prec(\cdots (b\prec
b))}_i)-x_i| i\in \mathbb{N}^+\},\\
F_4&=&\{f_{4(lm)}=a\prec (\underbrace{b\prec\cdots\prec (b}_m\prec
(\underbrace{a\prec\cdots \prec (a\prec a}_l))))-x_m^{(l)}| m,l\in
\mathbb{N}^+\},\\
F_5&=&\{f_{5(lm)}=a\prec (\underbrace{b\prec\cdots\prec (b}_m\succ
(\underbrace{a\succ\cdots \succ (a\succ a}_l))))-y_m^{(l)}| m,l\in
\mathbb{N}^+\},
\end{eqnarray*}
Let $F_{1-5}=F_1\cup F_2\cup F_3\cup F_4\cup F_5.$
\begin{eqnarray*}
F_{6(1)}&=&\{f_{6(1m)}=x_m^{(1)}\prec (f^{(0)}\succ
y_m^{(1)})-g^{(0)}\ |\ f^{(0)},g^{(0)}\in k(Irr(F_1\cup F_2)\cap L(X_0)),\\
&&\ \ \ \ f^{(0)} \mbox{ and } g^{(0)} \mbox{ are monic}\},\\
F_{7(1)}&=&\{f_{7(1ij)}=x_i\succ ((((x_j\prec u)\prec
v_1)\prec\cdots)\prec v_n )-\\
&&\ \ (((\{x_i\succ x_j\}\prec u)\prec v_1)\prec\cdots)\prec v_n\ |\
u\in Irr(F_{1-5}\cup F_{6(1)})\cap
L(X_1)-X_0,\\
&&\ \ \ \ v_r\in Irr(F_{1-5}\cup F_{6(1)})\cap L(X_1),\
r=1,\dots,n,\ n\in \mathbb{N},\},
\end{eqnarray*}

For $l\geq2$, let
\begin{eqnarray*}
F_{6(l)}&=&\{f_{6(lm)}=\underbrace{x_m^{(l)}\prec( \cdots\prec
(x_m^{(l)}}_{|\overline{g^{(l-1)}}|}\prec (f^{(l-1)}\succ
y_m^{(l)})))-g^{(l-1)}\ |\ f^{(l-1)},g^{(l-1)}\in \\
&&\ \  k(Irr(F_{1-5}\cup F_{6(1)}\cup \dots \cup F_{6(l-1)}\cup
F_{7(1)}\cup\dots\cup F_{7(l-1)})\cap
L(X_{l-1}))\ \mbox{ are monic}\},\\
F_{7(l)}&=&\{f_{7(lij)}=x_i\succ ((((x_j\prec u)\prec
v_1)\prec\cdots)\prec v_n )-\\
&&(((\{x_i\succ x_j\}\prec u)\prec v_1)\prec\cdots)\prec v_n\ |\
n\in \mathbb{N},\\
&&u\in Irr(F_{1-5}\cup F_{6(1)}\cup \dots \cup F_{6(l)}\cup
F_{7(1)}\cup\dots\cup F_{7(l-1)})\cap L(X_l)-X_0,\\
&&v_r\in Irr(F_{1-5}\cup F_{6(1)}\cup \dots \cup F_{6(l)}\cup
F_{7(1)}\cup\dots\cup F_{7(l-1)})\cap L(X_l),\ r=1,\dots,n\}.
\end{eqnarray*}

Let $F_{6}=\cup_{l\geqslant 1}F_{6(l)}$, $F_{7}=\cup_{l\geqslant
1}F_{7(l)}$.

Now let $B=L(X|S)$, where $S=F_{1-5}\cup F_{6}\cup F_{7}$.

We prove that $S$ is a Gr\"{o}bner-Shirshov basis in $L(X)$ with the
ordering (\ref{1}), where $X$ is a well-ordered set.

The only ambiguity $w$ of composition of inclusion in $S$ is
$\overline{f_{7(lij)}}|_{\overline{f'_{7(l'i'j')}}}$ where
$\overline{f'_{7(l'i'j')}}$ is a subword of $u$ or some $v_r$. It is
trivial by a similar proof in Theorem \ref{GSBl}.

The possible compositions of  right multiplication are:
$f_{1(ij)}\prec u$ and $f_{7(lij)}\prec u$, $u\in N(X),\
l\in\mathbb{N}^+$.

For $u\in N(X)$, $\exists l\in \mathbb{N}$ such that $u\in L(X_l)$.
We can find an $f=\sum_j\alpha_jw_j\in L(X_l)$, where $\alpha_j\in
k$, $w_j\in Irr(F_{1-5}\cup F_{6(1)}\cup \dots \cup F_{6(l)}\cup
F_{7(1)}\cup\dots\cup F_{7(l-1)})\cap L(X_l)$ such that $u\equiv f\
\ mod(S)$, and each $w_j\leq u$. So
$$
f_{1(ij)}\prec u\equiv f_{1(ij)}\prec f\equiv
\sum_j\alpha_jf_{1(ij)}\prec w_j\ \ mod(S).
$$
For each $w_j$, if $w_j\notin X_0$, then by $F_{7(l)}$,
$$
f_{1(ij)}\prec w_j=(x_i\succ x_j)\prec w_j- \{x_i\succ x_j\}\prec
w_j \equiv 0\ \ mod(S).
$$
If $w_j=x_k\in X_0$, then
$$
f_{1(ij)}\prec w_j\equiv 0\ \ mod(S).
$$
So all the compositions of right multiplication of $f_{1(ij)}$ is
trivial.

Now we consider the case of $f_{7(lij)}$. For any $w\in N(X)$,
$\exists l'\in \mathbb{N}$ such that $w\in L(X_{l'})$. If $l\geq
l'$, we can find an  $f=\sum_r\alpha_rw_r\in L(X_l)$, where
$\alpha_r\in k$, $w_r\in Irr(F_{1-5}\cup F_{6(1)}\cup \dots \cup
F_{6(l)}\cup F_{7(1)}\cup\dots\cup F_{7(l-1)})\cap L(X_l)$ such that
$w\equiv f\ \ mod(S)$, and $w_r\leq w$. So
$$
f_{7(lij)}\prec w\equiv f_{7(lij)}\prec f\equiv
\sum_r\alpha_rf_{7(lij)}\prec w_r\equiv0\ \ mod(S)
$$
since for each $w_r$, $f_{7(lij)}\prec w_r$ is of the form
$f_{7(lij)}$.

Suppose that $l<l'$. For any $t\in\{u, v_1,\dots,v_n,w\}$, we can
find an  $f=\sum_{tr}\alpha_{tr}w_{tr}\in L(X_{l'})$, where
$\alpha_{tr}\in k$, $w_{tr}\in Irr(F_{1-5}\cup F_{6(1)}\cup \dots
\cup F_{6(l')}\cup F_{7(1)}\cup\dots\cup F_{7(l'-1)})\cap L(X_{l'})$
such that $t\equiv f\ \ mod(S)$, and $w_{tr}\leq t$. Then
$f_{7(lij)}\prec w\equiv0\ \ mod(S)$.

So $S$ is a Gr\"{o}bner-Shirshov basis in $L(X)$. By Theorem
\ref{cd}, $A$ can be embedded into $B=L(X|S)$. By $F_3$, $F_4$ and
$F_5$, $B$ is generated by $\{a, b\}$. Note that for every non-zero
element $f$ in $B$, there exists a monic polynomial $f^{(l)}\in
k(Irr(F_{1-5}\cup F_{6(1)}\cup \dots \cup F_{6(l)}\cup
F_{7(1)}\cup\dots\cup F_{7(l)})\cap L(X_{l}))$ such that $\alpha
f^{(l)}= f$ where $\alpha\in k$. Then by $F_6$, every two non-zero
elements in $B$ generate the same ideal. Thus, $B$ is simple.

The proof is completed.
 \ \hfill \end{proof}

\begin{theorem}
$A$, $B$, $C$ are arbitrary $L$-algebras over a field $k$. If
$|k|\leq \dim(B*C)$ and $|A|\leq|B*C|$, where $B*C$ is the free
product of $B$ and $C$, then $A$, $B$, $C$ can be embedded into a
simple $L$-algebra generated by $B$ and $C$.
\end{theorem}

\begin{proof}
We consider firstly the case that $B$ and $C$ are
finite-dimensional. In this case, $|k|\leq\aleph_0$ and
$|A|\leq\aleph_0$.

We may assume that $A$ has a countable $k$-basis $X_A=\{a_i|i\in
I_A\}$, $B$ has a finite $k$-basis $X_B=\{b_i|i\in I_B\}$ and $C$
has a finite $k$-basis $X_C=\{c_i|i\in I_C\}$. By Lemma \ref{eml},
$A=L(X_A|S_A)$, where $S_A=\{a_i\prec a_j-\{a_i\prec a_j\},\
a_i\succ a_j-\{a_i\succ a_j\}|i,\ j \in I_A\}$, $B=L(X_B|S_B)$,
where $S_B=\{b_i\prec b_j-\{b_i\prec b_j\},\ b_i\succ b_j-\{b_i\succ
b_j\}|i,\ j \in I_B\}$, $C=L(X_C|S_C)$, where $S_C=\{c_i\prec
c_j-\{c_i\prec c_j\},\ c_i\succ c_j-\{c_i\succ c_j\}|i,\ j \in
I_C\}$.

Let $X_0=X_A\cup X_B\cup X_C$, $L(X_0)^+=L(X_0)-\{0\}$ and fix the
bijection
$$
(L(X_0)^+,L(X_0)^+)\leftrightarrow\{(x_m^{(1)},y_m^{(1)})|m\in
\mathbb{N}^+\}.
$$

For $n\geqslant 0$, let
$X_{n+1}=X_n\cup\{x_m^{(n+1)},y_m^{(n+1)}|m\in \mathbb{N}^+\}$,
$L(X_{n+1})^+=L(X_{n+1})-\{0\}$ and fix the bijection
$$
(L(X_{n+1})^+,L(X_{n+1})^+)\leftrightarrow\{(x_m^{(n+2)},y_m^{(n+2)})|m\in
\mathbb{N}^+\}.
$$

Let $X=\cup_{n=0}^{\infty}X_n$. Then
$L(X)=\cup_{n=0}^{\infty}L(X_{n})$.

Note that $|X_A\cup\{x_m^{(n)},y_m^{(n)}|m\in \mathbb{N}^+,\ n\in
\mathbb{N}\}|=\aleph_0$. Define
\begin{eqnarray*}
F_{1a}&=&\{f_{1a(ij)}=a_i\succ a_j-\{a_i\succ a_j\}| i,j\in  I_A\},\\
F_{1b}&=&\{f_{1b(ij)}=b_i\succ b_j-\{b_i\succ b_j\}| i,j\in  I_B\},\\
F_{1c}&=&\{f_{1c(ij)}=c_i\succ c_j-\{c_i\succ c_j\}| i,j\in  I_C\},\\
F_{2a}&=&\{f_{2a(ij)}=a_i\prec a_j-\{a_i\prec a_j\}| i,j\in  I_A\},\\
F_{2b}&=&\{f_{2b(ij)}=b_i\prec b_j-\{b_i\prec b_j\}| i,j\in  I_B\},\\
F_{2c}&=&\{f_{2c(ij)}=c_i\prec c_j-\{c_i\prec c_j\}| i,j\in  I_C\},\\
F_{3}&=&\{f_{3(i)}=\underbrace{(((b_0\prec c_0)\prec(b_0\prec
c_0))\prec\cdots\prec(b_0\prec c_0) )}_i\succ(b_0\prec c_0)-d_i\ |\
i\in \mathbb{N}^+\},
\end{eqnarray*}
where $ b_0\in X_B,\ c_0 \in X_C$ are two fixed elements and
$X_A\cup\{x_m^{(n)},y_m^{(n)}|m\in \mathbb{N}^+,\ n\in
\mathbb{N}\}=\{d_i|i\in \mathbb{N}^+\}$.

Let $F_{1-3}=F_{1a}\cup F_{1b}\cup F_{1c}\cup F_{2a}\cup F_{2b}\cup
F_{2c}\cup F_{3}$.
\begin{eqnarray*}
F_{4(1)}&=&\{f_{4(1m)}=x_m^{(1)}\prec (f^{(0)}\succ
y_m^{(1)})-g^{(0)}\ |\ f^{(0)},g^{(0)}\in\\
&&\ \ k(Irr(F_{1a}\cup F_{1b}\cup F_{1c}\cup F_{2a}\cup F_{2b}\cup
F_{2c})\cap L(X_0))\mbox{ are monic}\},\\
F_{5(a1)}&=&\{f_{5(a1ij)}=a_i\succ ((((a_j\prec u)\prec
v_1)\prec\cdots)\prec v_n )-\\
&&\ \ (((\{a_i\succ a_j\}\prec u)\prec v_1)\prec\cdots)\prec v_n\ |\
u\in Irr(F_{1-3}\cup F_{4(1)})\cap L(X_1)-X_A,\\
&&\ \ v_r\in Irr(F_{1-3}\cup F_{4(1)})\cap L(X_1),\ r=1,\dots,n,\
n\in \mathbb{N}\},\\
F_{5(b1)}&=&\{f_{5(b1ij)}=b_i\succ ((((b_j\prec u)\prec
v_1)\prec\cdots)\prec v_n )-\\
&&\ \ (((\{b_i\succ b_j\}\prec u)\prec v_1)\prec\cdots)\prec v_n\ |\
u\in Irr(F_{1-3}\cup F_{4(1)})\cap L(X_1)-X_B,\\
&&\ \ v_r\in Irr(F_{1-3}\cup F_{4(1)})\cap L(X_1),\ r=1,\dots,n,\
n\in \mathbb{N}\},\\
F_{5(c1)}&=&\{f_{5(c1ij)}=c_i\succ ((((c_j\prec u)\prec
v_1)\prec\cdots)\prec v_n )-\\
&&\ \ (((\{c_i\succ c_j\}\prec u)\prec v_1)\prec\cdots)\prec v_n\ |\
u\in Irr(F_{1-3}\cup F_{4(1)})\cap L(X_1)-X_C,\\
&&\ \ v_r\in Irr(F_{1-3}\cup F_{4(1)})\cap L(X_1),\ r=1,\dots,n,\
n\in \mathbb{N}\},\\
F_{5(31)}&=&\{f_{5(31i)}=\underbrace{(((b_0\prec c_0)\prec(b_0\prec
c_0))\prec\cdots\prec(b_0\prec c_0) )}_i\\
&&\ \ \succ (((((b_0\prec c_0)\prec u)\prec v_1)\prec\cdots)\prec
v_n ) -(((d_i\prec u)\prec v_1)\prec\cdots)\prec v_n\ |\\
&&\ \ u,v_r\in Irr(F_{1-3}\cup F_{4(1)})\cap L(X_1),\ r=1,\dots,n,\
n\in \mathbb{N}\}.
 \end{eqnarray*}

Denote $T_{l}=F_{1-3}\cup F_{4(1)}\cup \dots\cup F_{4(l)}\cup
F_{5(a1)}\cup \dots\cup F_{5(al)}\cup F_{5(b1)}\cup \dots\cup
F_{5(bl)}\cup F_{5(c1)}\cup \dots\cup F_{5(cl)}\cup F_{5(31)}\cup
\dots\cup F_{5(3l)}.$ Then $T_1$ is defined. We use induction to
define $T_l$.

Assume $l\geq2$. Let
 \begin{eqnarray*}
F_{4(l)}&=&\{f_{4(lm)}=\underbrace{x_m^{(l)}\prec( \cdots\prec
(x_m^{(l)}}_{| \overline{g^{(l-1)}}|}\prec (f^{(l-1)}\succ
y_m^{(l)})))-g^{(l-1)}\ |\\
&&\ \ f^{(l-1)},g^{(l-1)}\in k(Irr(T_{l-1})\cap L(X_{l-1}))
\mbox{ are monic}\},\\
F_{5(al)}&=&\{f_{5(alij)}=a_i\succ ((((a_j\prec u)\prec
v_1)\prec\cdots)\prec v_n )-\\
&&\ \ (((\{a_i\succ a_j\}\prec u)\prec v_1)\prec\cdots)\prec v_n\ |\
u\in Irr(T_{l-1}\cup F_{4(l)})\cap L(X_l)-X_A,\\
&&\ \ v_r\in Irr(T_{l-1}\cup F_{4(l)})\cap L(X_l),\ r=1,\dots,n,\
n\in \mathbb{N}\},\\
F_{5(bl)}&=&\{f_{5(blij)}=b_i\succ ((((b_j\prec u)\prec
v_1)\prec\cdots)\prec v_n )-\\
&&\ \ (((\{b_i\succ b_j\}\prec u)\prec v_1)\prec\cdots)\prec v_n\ |\
u\in Irr(T_{l-1}\cup F_{4(l)})\cap L(X_l)-X_B,\\
&&\ \ v_r\in Irr(T_{l-1}\cup F_{4(l)})\cap L(X_l),\ r=1,\dots,n,\
n\in \mathbb{N}\},\\
F_{5(cl)}&=&\{f_{5(clij)}=c_i\succ ((((c_j\prec u)\prec
v_1)\prec\cdots)\prec v_n )-\\
&&\ \ (((\{c_i\succ c_j\}\prec u)\prec v_1)\prec\cdots)\prec v_n\ |\
u\in Irr(T_{l-1}\cup F_{4(l)})\cap L(X_l)-X_C,\\
&&\ \ v_r\in Irr(T_{l-1}\cup F_{4(l)})\cap L(X_l),\ r=1,\dots,n,\
n\in \mathbb{N}\},\\
F_{5(3l)}&=&\{f_{5(3li)}=\underbrace{(((b_0\prec c_0)\prec(b_0\prec
c_0))\prec\cdots\prec(b_0\prec c_0) )}_i\\
&&\ \ \succ (((((b_0\prec c_0)\prec u)\prec v_1)\prec\cdots)\prec
v_n ) -(((d_i\prec u)\prec v_1)\prec\cdots)\prec v_n\ |\\
&&\ \ u,v_r\in Irr(T_{l-1}\cup F_{4(l)})\cap L(X_l),\ r=1,\dots,n,\
n\in \mathbb{N}\}.
 \end{eqnarray*}

Let $F_{4}=\cup_{l\geqslant 1} F_{4(l)}$,
$F_5=\cup_{i=a,b,c,3}(\cup_{l\geqslant 1} F_{5(il)})$. Now let
$D=L(X|S)$ where $S=F_{1-3}\cup F_{4}\cup F_{5}$.

By a similar proof in Theorem \ref{em4}, $S$ is a
Gr\"{o}bner-Shirshov basis in $L(X)$ with the ordering (\ref{1}),
where $X$ is a well-ordered set.  By Theorem \ref{cd}, $A$, $B$ and
$C$ can be embedded into $D=L(X|S)$. By $F_{3}$, $D$ is generated by
$B$ and $C$. By $F_{4}$, $D$ is simple. Now, we finish the proof in
the case of finite-dimensional $L$-algebras $B$ and $C$.

Consider secondly the case that one of $B$ and $C$ is infinite
dimensional. Assume that $\dim B\leq\dim C$, $\dim
C=\alpha\geq\aleph_0$. Since $|k|\leq \dim(B*C)$, so
$\dim(B*C)=|B*C|=\alpha$. Without loss of generality, we can assume
that $\dim A=\alpha$ and so $|A|=\alpha$.

In this case, we only need to change $F_{3}$ into
 \begin{eqnarray*}
F_3&=&\{f_{3(i\beta)}=\underbrace{(((b_0\prec c_0)\prec(b_0\prec
c_0))\prec\cdots\prec(b_0\prec c_0) )}_i\succ(b_0\prec
c_\beta)-d_{i\beta}\ |\\
&&\ \ i\geq 1,\ 1\leq \beta<\alpha\},
 \end{eqnarray*}
where  $X_A\cup\{x_m^{(n)},y_m^{(n)}|1\leq m<\alpha,\ n\geq 1
\}=\{d_{i\beta}|i\geq 1,\ 1\leq \beta<\alpha\}$ with cardinal number
$\alpha$, and $\{c_{\beta}|1\leq \beta<\alpha\}\subseteq X_C$ such
that there is a one-to-one correspondence $\{(i,c_{\beta})|i\in
\mathbb{N}^+,\ 1\leq \beta<\alpha\}\leftrightarrow \{d_{i\beta}|i\in
\mathbb{N}^+,\ 1\leq \beta<\alpha\}$.

Now, in order to keep $S$ as a Gr\"{o}bner-Shirshov basis, we change
$F_{5(3l)}$, $l\in \mathbb{N}^+$ into
 \begin{eqnarray*}
F_{5(3l)}&=&\{f_{5(3li\beta)}=\underbrace{(((b_0\prec
c_0)\prec(b_0\prec
c_0))\prec\cdots\prec(b_0\prec c_0) )}_i\\
&&\ \ \succ (((((b_0\prec c_\beta)\prec u)\prec
v_1)\prec\cdots)\prec
v_n ) -(((d_{i\beta}\prec u)\prec v_1)\prec\cdots)\prec v_n\ |\\
&&\ \ u,v_r\in Irr(T_{l-1}\cup F_{4(l)})\cap L(X_l),\ r=1,\dots,n,\
n\in \mathbb{N},\ i\geq 1,\ 1\leq \beta<\alpha\}.
 \end{eqnarray*}

Then by the same analysis,  $A$, $B$ and $C$ can be embedded into
$D=L(X|S)$ which is a simple $L$-algebra generated by $B$ and $C$.

The proof is completed.
 \ \hfill \end{proof}

\end{document}